\documentclass{amsart}

\usepackage[english]{babel}
\usepackage{graphicx}
\usepackage{amsmath,amssymb,hyperref}
% ...

% ...

\newcommand{\diff}[2]{\mbox{{\rm Diff}{${\,}_{#1}({\mathbb C}^{#2},0)$}}}
\newcommand{\diffh}[2]{\mbox{$\widehat{\rm Diff}{{\,}_{#1}({\mathbb C}^{#2},0)}$}}

\newcommand{\cn}[1]{\mbox{(${\mathbb C}^{#1},0$)}}

\newtheorem*{main}{Main Theorem}
\newtheorem{pro}{Proposition}[section]

\newtheorem{cor}{Corollary}[section]

\newtheorem{lem}{Lemma}[section]
\newtheorem{rem}{Remark}[section]
\newtheorem{defi}{Definition}[section]
\title
[Non embeddability of unipotent diffeomorphisms]
{Non-embeddability of general unipotent diffeomorphisms up to
formal conjugacy}

\author{Javier Rib\'{o}n}

\address{Instituto de Matem\'{a}tica, UFF, Rua M\'{a}rio Santos Braga S/N
Valonguinho, Niter\'{o}i, Rio de Janeiro, Brasil 24020-140}

\email{javier@mat.uff.br}

\thanks{The work of the author has been partly financed by Junta de Castilla y Le\'{o}n, Project
VA059A07 and CNPQ, Project PRONEX-Teoria Geom\'{e}trica das Equa\c{c}\~{o}es Diferenciais Complexas}



\begin{document}
%% Abstract
\begin{abstract}
 The formal class of a germ of diffeomorphism $\varphi$ is embeddable in a flow
if $\varphi$ is formally conjugated to the exponential of a germ
of vector field. We prove that there are complex analytic
unipotent germs of diffeomorphisms at $\cn{n}$ ($n>1$) whose
formal class is non-embeddable. The examples are inside a family
in which the non-embeddability is of geometrical type. The proof
relies on the properties of some linear functional operators that
we obtain through the study of polynomial families of
diffeomorphisms via potential theory.
\end{abstract}

\maketitle

\section{Introduction}
In this paper we prove
\begin{main}
There exists a unipotent germ of complex analytic diffeomorphism at $\cn{2}$
whose formal class is non-embeddable.
\end{main}
Denote by $\diff{}{n}$ the set of germs of complex analytic diffeomorphisms at $\cn{n}$ whereas
$\diffh{}{n}$ is the formal completion of $\diff{}{n}$.
A ``normal form" for $\varphi \in \diff{}{n}$ should be a diffeomorphism formally conjugated to $\varphi$
but somehow simpler. Every $\varphi \in \diff{}{n}$ admits a unique Jordan decomposition
\[ \varphi = \varphi_{s} \circ \varphi_{u} = \varphi_{u} \circ \varphi_{s} \]
where $\varphi_{s} \in \diffh{}{n}$ is semisimple and $\varphi_{u} \in \diffh{}{n}$ is unipotent,
that is $j^{1} \varphi_{u} - Id$ is nilpotent.
Then $\varphi_{s}$ is formally linearizable and has the natural normal form $j^{1} \varphi_{s}$.
Note that $\varphi_{u}$
is not formally linearizable unless $\varphi_{u} \equiv Id$. Thus, a different approach is required
to obtain simple models, up to formal conjugacy, for unipotent diffeomorphisms.

  A unipotent $\varphi \in \diff{}{n}$ is the exponential of a unique formal
nilpotent vector field $X$, the so called infinitesimal generator of $\varphi$
(see \cite{Ecalle} and \cite{MaRa:aen}), which is
geometrically significant even if it diverges in general. We denote $X$ by $\log \varphi$.
We say that the {\it formal class of $\varphi$ is embeddable} if $\log \varphi$ is formally conjugated to
a germ of convergent vector field $Y$.
The diffeomorphisms of the form ${\rm exp}(Y)$ provide continuous models
${\rm exp}(t Y)$ ($t \in {\mathbb C}$) to compare with the
discrete dynamics of $\varphi$.

  The existence of embeddable elements in the formal class of a diffeomorphism has useful implications.
For instance, for $n=1$, a unipotent $\varphi \in \diff{}{}$ satisfies $j^{1} \varphi =Id$ and
$\log \varphi$ is always conjugated to a germ of analytic vector field $Y$. The normalizing
transformation is in general divergent. Anyway, there exist regions (Fatou petals) where
it is the asymptotic development of an analytic mapping conjugating ${\rm exp}(Y)$ and $\varphi$.
This phenomenon provides the basis (see Il'yashenko's paper in \cite{Stokes:Ilya}) to construct the
Ecalle-Voronin complete system of analytic invariants
(Ecalle \cite{E}, Voronin \cite{V}, Martinet-Ramis \cite{MaRa:aen}, Malgrange \cite{mal:ast}).
The same strategy can be applied in some cases in higher dimension.
For instance Voronin \cite{Stokes:Voronin} classifies analytically the unipotent
diffeomorphisms in $\diff{}{2}$ which are formally conjugated to a diffeomorphism of the form
$(x,y + {x}^{k})={\rm exp}({x}^{k} \partial / \partial{y})$ for some $k \in {\mathbb N}$.

In the previous cases the analytic classification relies on the study of the normalizing transformation
to a simpler model.
Normalizing transformations are quite well understood whereas it is more difficult to describe
the possible final models up to formal conjugacy and their properties.
A good example is provided by Birkhoff normal forms attached to
analytic Hamiltonian vector fields (see \cite{PM1}, see \cite{Kuksin} for properties of
embeddability of symplectic diffeomorphisms in the flow of an analytic Hamiltonian vector field).
On the one hand the properties of the normalizing
transformations are well-documented. On the other hand the Birkhoff normal forms are
either always convergent or generically divergent \cite{PM1}. Nevertheless P\'{e}rez Marco
stresses that whether or not there exists a divergent Birkhoff normal form
is still an open problem. Our Main Theorem claims that
examples of unipotent diffeomorphisms whose formal class only contains diffeomorphisms
with divergent infinitesimal generator can be provided.
They can be chosen of the form
\[ \varphi_{\Delta,w}(x,y) = (x+y(y-x) \Delta(x,y), y+y(y-x)w(x,y)) . \]
Moreover, for this kind of diffeomorphisms
the obstruction to the embeddability of the formal class is of geometrical type.

  The proof of the Main Theorem is based on the study of the transport mapping. Suppose
$\log \varphi_{\Delta,w}$ is a germ of convergent vector field. Then
\[ L_{\Delta,w} \stackrel{def}{=} \frac{1}{y(y-x)} \log \varphi_{\Delta,w} \]
is regular, i.e. $L_{\Delta,w}(0) \neq 0$, and transversal to both $y=0$ and $y=x$.
\begin{figure}
\begin{center}
\includegraphics[height=5cm,width=9cm]{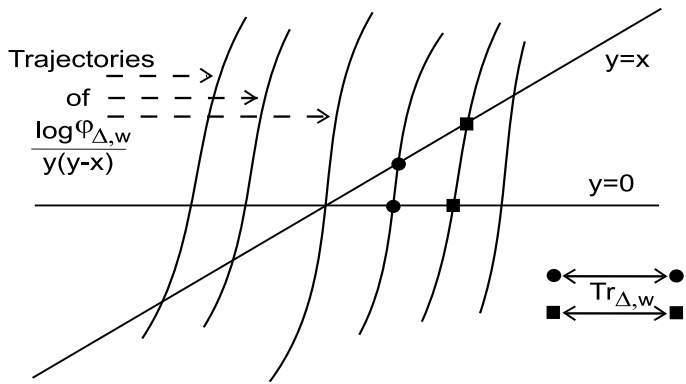}
\end{center}
\begin{center}
\caption{Real picture of the transport mapping}
\end{center}
\end{figure}
We can define a correspondence $Tr_{\Delta,w}$ associating to each point $P$ in $y=0$
the unique point in $y=x$ contained in the trajectory of $L_{\Delta,w}$
passing through $P$. This correspondence is the transport mapping. Even if
$\log \varphi_{\Delta,w}$ is divergent we manage to define $Tr_{\Delta,w}$; it is a formal invariant.
We prove the following characterization for the non-embeddability of the formal class:
\begin{pro}
If the formal class of $\varphi_{\Delta,w}$ is embeddable then $Tr_{\Delta,w}$ is an analytic mapping.
\end{pro}
  Fix a unit $w \in {\mathbb C}\{x,y\}$, $w(0) \neq 0$,  such that $\ln \varphi_{0,w}$ is not
convergent. We claim that there exists
$\Delta$ in ${\mathbb C}\{x,y\}$, $\Delta(0)=0$,  such that $Tr_{\Delta,w}$ diverges. We argue by contradiction.
Note that $Tr_{\Delta,w}$ can be analytic
(for example $Tr_{\Delta,w}(x,0) \equiv (x,x)$) whereas $\log \varphi_{\Delta,w}$ is divergent.
The divergence of $Tr_{\Delta,w}$ is even stronger than the non-existence of
a germ of convergent vector field collinear to $\log \varphi_{\Delta,w}$. It is obtained
through a fine analysis of the nature of the family ${(\varphi_{\Delta,w})}_{\Delta}$.

  We consider polynomial families on $\lambda \in {\mathbb C}$ of the form
\[ \varphi_{\lambda \Delta,w}(x,y) = (x + \lambda y(y-x) \Delta(x,y) , y+y(y-x) w(x,y) ) . \]
The embeddability of the formal class of $\varphi_{\lambda \Delta,w}$
for any $\lambda \in {\mathbb C}$ allows to find a linear equation
\[ \hat{\epsilon} - \hat{\epsilon} \circ \varphi_{0,w} = y(y-x) \Delta
\ \ \ \ {\rm (homological \ equation)}  \]
such that $\hat{\epsilon}_{\Delta}(x,x) - \hat{\epsilon}_{\Delta}(x,0) \in {\mathbb C}\{x\}$ for every
solution $\hat{\epsilon}_{\Delta} \in {\mathbb C}[[x,y]]$.
The proof of the convergence of $\hat{\epsilon}_{\Delta}(x,x) - \hat{\epsilon}_{\Delta}(x,0)$ is based
on potential theory techniques.
The homological equation has a formal
solution $\hat{\epsilon}_{\Delta}$ for any $\Delta \in {\mathbb C}[[x,y]]$, moreover
$\hat{\epsilon}_{\Delta}(x,x) - \hat{\epsilon}_{\Delta}(x,0)$ does not depend on the choice
of $\hat{\epsilon}_{\Delta}$. Then the operator
$S_{w}:{\mathbb C}[[x,y]] \to {\mathbb C}[[x]]$ given by
\[ S_{w}(\Delta) = \hat{\epsilon}_{\Delta}(x,x) - \hat{\epsilon}_{\Delta}(x,0) \]
is linear, well-defined and $S_{w}({\mathfrak m}_{x,y}) \subset {\mathbb C}\{x\}$
where ${\mathfrak m}_{x,y} \subset {\mathbb C}\{x,y\}$ is the maximal ideal.
Now it suffices to study a linear operator attached
to the dynamically simple diffeomorphism $\varphi_{0,w}$, in particular
$x \circ \varphi_{0,w} = x$ will be a key property to prove that
$S_{w}({\mathfrak m}_{x,y})$ contains divergent elements.

  The operator $S_{w}$ was defined in terms of a difference equation. We can replace the
difference equation with a differential equation easier to handle.
More precisely
$S_{w}(\Delta) = \hat{\Gamma}(x,x) - \hat{\Gamma}(x,0)$ for
every solution $\hat{\Gamma} \in {\mathbb C}[[x,y]]$ of
\[ (\log \varphi_{0,w}) (\hat{\Gamma}) = - y (y-x) \Delta . \]
Since $\log \varphi_{0,w}$ diverges and  $x \circ \varphi_{0,w} = x$ then
$\log \varphi_{0,w}= \hat{w} y (y-x) \partial / \partial{y}$ for
some divergent $\hat{w} \in {\mathbb C}[[x,y]]$. The collinearity of
$\log \varphi_{0,w}$ and $\partial/\partial{y}$, in addition to some
functional analysis techniques, can be used to prove that
the property $S_{w}({\mathfrak m}_{x,y}) \subset {\mathbb C}\{x\}$ implies that
$\hat{w} \in {\mathbb C}\{x,y\}$. Here we have our contradiction.

  We do not describe the nature of the transport mapping, besides the
fact that it is generically divergent. It would be interesting to know what
divergent mappings can be obtained as transport mappings of diffeomorphisms of
type $\varphi_{\Delta,w}$.
\section{Basic facts}
\label{sec:basic}
In this section we introduce some well-known facts about diffeomorphisms and vector
fields for the sake of completeness.

We denote by ${\mathcal X} \cn{n}$ the set of germs of complex analytic vector fields
at $0 \in  {\mathbb C}^{n}$.  A {\it formal vector field} $\hat{X}$ is a derivation of the ring
${\mathbb C}[[x_{1},\hdots,x_{n}]]$. We also express $\hat{X}$ in the more conventional form
\[ \hat{X} = \sum_{j=1}^{n} \hat{X}(x_{j}) \frac{\partial}{\partial x_{j}}  . \]
Given $X \in {\mathcal X} \cn{n}$ the derivation associated to $X$ restricts to
a derivation of the ring ${\mathbb C} \{ x_{1},\hdots,x_{n} \}$. Indeed
$X(g)$ is the Lie derivative $L_{X} g$ for any $g \in {\mathbb C}\{x_{1}, \hdots, x_{n}\}$.

We say that a formal vector field
$\hat{X} = \sum_{j=1}^{n} \hat{a}_{j}(x_{1}, \hdots, x_{n}) \partial / \partial{x_{j}}$ where
$\hat{a}_{j} \in {\mathbb C}[[x_{1}, \hdots, x_{n}]]$ for any
$1 \leq j \leq n$ is nilpotent if $\hat{X}(0)=0$ and $j^{1} \hat{X}$ is nilpotent. We denote by
$\hat{\mathcal X}_{N} \cn{n}$ and ${\mathcal X}_{N} \cn{n}$ the sets of formal nilpotent vector fields
and germs of nilpotent vector fields respectively.

Let $\hat{\mathfrak m}$ be the maximal ideal of ${\mathbb C}[[x_{1},\hdots,x_{n}]]$.
We say that a {\it formal diffeomorphism} is an automorphism
$\hat{\sigma}: {\mathbb C}[[x_{1}, \hdots, x_{n}]] \to {\mathbb C}[[x_{1}, \hdots, x_{n}]]$ of
${\mathbb C}$-algebras such that $\hat{\sigma}(\hat{\mathfrak m})= \hat{\mathfrak m}$.
Equivalently we can express $\hat{\sigma}$ in the form
\[ \hat{\sigma} = (\hat{\sigma}(x_{1}) , \hdots, \hat{\sigma}(x_{n})) \in {\mathbb C}[[x_{1}, \hdots, x_{n}]]^{n} \]
where $j^{1} \hat{\sigma}$ is a linear isomorphism. We denote by $\diffh{}{n}$ and
$\diff{}{n}$ the set of formal diffeomorphisms and germs of diffeomorphisms respectively.
If $j^{1} \hat{\sigma}$ is unipotent
(i.e. if $1$ is the only eigenvalue of $j^{1} \hat{\sigma}$) then we say that $\hat{\sigma}$ is unipotent.
We denote by $\diffh{u}{n}$ the set of formal unipotent diffeomorphisms.

Let $X \in {\mathcal X} \cn{n}$; suppose that $X$ is singular at $0$.
We denote by ${\rm exp}(tX)$ (for $t \in {\mathbb C}$) the flow of the vector field $X$, it is
the unique solution
of the differential equation
\[ \frac{\partial}{\partial{t}} {\rm exp}(tX) = X({\rm exp}(tX)) \]
with initial condition ${\rm exp}(0X)=Id$. We define the exponential
${\rm exp}(X)$ of $X$ as
${\rm exp}(1X)$. We define the exponential operator
\[
\begin{array}{rccc}
{\rm exp}(\hat{X}): & {\mathbb C}[[x_{1},\hdots,x_{n}]] & \to & {\mathbb C}[[x_{1},\hdots,x_{n}]]
\\
& g & \to & \sum_{j=0}^{\infty} \frac{\hat{X}^{j}}{j!} (g) .
\end{array}
\]
for any element $\hat{X}$  of $\hat{\mathcal X}_{N} \cn{n}$.
The definitions of exponential coincide if $\hat{X} \in {\mathcal X}_{N} \cn{n}$, i.e. the image of
$g$ by the operator ${\rm exp}(\hat{X})$ is $g \circ {\rm exp}(\hat{X})$ for any
$g \in {\mathbb C} \{x_{1},\hdots,x_{n}\}$.
Given $t \in {\mathbb C}$ the formal flow ${\rm exp}(t \hat{X})$ is an operator such that
\begin{equation}
\label{equ:fflow}
{\rm exp}(t \hat{X})(g) =  \sum_{j=0}^{\infty} \frac{(t \hat{X})^{j}}{j!} (g) =
\sum_{j=0}^{\infty} t^{j} \frac{\hat{X}^{j}}{j!} (g)
\end{equation}
for any $g \in {\mathbb C}[[x_{1},\hdots,x_{n}]]$. The property
\begin{equation}
\label{equ:comm}
{\rm exp}((t+s) \hat{X}) = {\rm exp}(t \hat{X}) \circ {\rm exp}(s \hat{X}) \ \forall s,t \in {\mathbb C}
\end{equation}
follows from equation (\ref{equ:fflow}) and the formal identity $e^{t+s}=e^{t} e^{s}$ since
$t \hat{X}$ and $s \hat{X}$ commute.
%\[ g \circ {\rm exp}(t \hat{X})  \circ {\rm exp}(s \hat{X}) =
%\sum_{k=0}^{\infty}  \frac{s^{k}}{k!} \hat{X}^{k} \left({
%\sum_{j=0}^{\infty} \frac{t^{j}}{j!} \hat{X}^{j}(g) }\right) =
%\sum_{j,k \geq 0}^{\infty} \frac{t^{j} s^{k}}{j! k!} \hat{X}^{j+k}(g) \]
%\[ \sum_{j,k \geq 0}^{\infty} \frac{t^{j} s^{k}}{j! k!} \hat{X}^{j+k}(g) =
%\sum_{l=0}^{\infty} \hat{X}^{l}(g) \sum_{j+k=l} \frac{t^{j} s^{k}}{j! k!} =
%\sum_{l=0}^{\infty} \frac{(t+s)^{l}}{l!}\hat{X}^{l}(g)  \]
%\[ \sum_{l=0}^{\infty} \frac{(t+s)^{l}}{l!}\hat{X}^{l}(g)  = g \circ {\rm exp}((t+s) \hat{X}) \]

The nilpotent character of $\hat{X}$ implies that the power series
${\rm exp}(\hat{X})(g)$ converges in the Krull
topology for any $g \in {\mathbb C}[[x_{1},\hdots,x_{n}]]$. Moreover, since $\hat{X}$
is a derivation
then ${\rm exp}(\hat{X})$ acts as a diffeomorphism, i.e.
\[ {\rm exp}(\hat{X})(g_{1} g_{2})= {\rm exp}(\hat{X})(g_{1}) {\rm exp}(\hat{X})(g_{2})
\ \ \forall  g_{1},g_{2} \in {\mathbb C}[[x_{1},\hdots,x_{n}]]. \]
%Then we can use the more conventional notation
%\[ {\rm exp}(\hat{X}) = \left({
%\sum_{j=0}^{\infty} \frac{\hat{X}^{\circ(j)}}{j!} (x_{1}), \hdots,
%\sum_{j=0}^{\infty} \frac{\hat{X}^{\circ(j)}}{j!} (x_{n-1}),
%\sum_{j=0}^{\infty} \frac{\hat{X}^{\circ(j)}}{j!} (y)
% }\right) .\]
Moreover $j^{1} {\rm exp}(\hat{X})= {\rm exp}(j^{1} \hat{X})$, thus
$j^{1} {\rm exp}(\hat{X})$ is a unipotent linear
isomorphism. The following proposition is classical.
\begin{pro} (see \cite{Ecalle}, \cite{MaRa:aen})
\label{pro:exp} The mapping $\hat{X} \mapsto {\rm
exp}(1 \hat{X})$ maps bijectively $\hat{\mathcal X}_{N} \cn{n}$
onto $\diffh{u}{n}$. Moreover, if $\hat{X} \in \hat{\mathcal
X}_{N} \cn{n}$ then ${\rm exp}(t \hat{X})(g)$ belongs to
${\mathbb C}[t][[x_{1} , \hdots , x_{n}]]$ for any $g \in {\mathbb C}[[x_{1},\hdots,x_{n}]]$.
\end{pro}
Consider the inverse mapping $\log: \diffh{u}{n} \to \hat{\mathcal X}_{N} \cn{n}$.
An element $\varphi$ of $\diffh{u}{n}$ defines an isomorphism
\[ \varphi : {\mathbb C}[[x_{1}, \hdots, x_{n}]] \to {\mathbb C}[[x_{1}, \hdots, x_{n}]] \]
of ${\mathbb C}$-algebras such that $\varphi(g) = g \circ \varphi$ for any
$g \in {\mathbb C}[[x_{1}, \hdots, x_{n}]]$. Denote by
$\Theta: {\mathbb C}[[x_{1}, \hdots, x_{n}]] \to {\mathbb C}[[x_{1}, \hdots, x_{n}]]$ the operator
$\varphi - Id$,  i.e. we have $\Theta(g) = \varphi(g) - Id(g) = g \circ \varphi -g$ for any
$g \in {\mathbb C}[[x_{1},\hdots,x_{n}]] $. The operator $\varphi - Id$ is not associated to
a diffeomorphism. We have
\begin{equation}
\label{equ:log}
 (\log \varphi)(g) = (\log (Id + \Theta))(g) =
\sum_{j=1}^{\infty} {(-1)}^{j+1} \frac{\Theta^{j} (g)}{j}
\end{equation}
for $g \in {\mathbb C}[[x_{1},\hdots,x_{n}]] $. The series in the
right hand side converges in the Krull topology since $\varphi$ is
unipotent. Moreover $j^{1}(\log \varphi)= \log(j^{1} \varphi)$ is
nilpotent and $\log \varphi$ satisfies the Leibnitz rule. We say
that $\log \varphi$ is the {\it infinitesimal generator} of
$\varphi$. Even if $\varphi \in \diffh{u}{n} \cap \diff{}{n}$ in general $\varphi$ is divergent.

  Consider an ideal $\hat{I} \subset  {\mathbb C}[[x_{1}, \hdots, x_{n}]]$. We denote by
$Z(\hat{I})$ the set of formal curves $\hat{\gamma} \in (t {\mathbb C}[[t]])^{n}$ such that
$\hat{h} \circ \hat{\gamma} =0$ for any $\hat{h} \in \hat{I}$. Conversely, for
$\hat{\Delta} \subset (t {\mathbb C}[[t]])^{n}$ we define $I(\hat{\Delta})$ as the set of
series $\hat{h} \in {\mathbb C}[[x_{1}, \hdots, x_{n}]]$ such that $\hat{h} \circ \hat{\gamma}=0$
for any $\hat{\gamma} \in \hat{\Delta}$. We have
\begin{pro}[Formal theorem of zeros \cite{Tou:ide}, pages 49-50]
Let $\hat{I}$ be an ideal of ${\mathbb C}[[x_{1}, \hdots, x_{n}]]$. Then
\[ I(Z(\hat{I})) = \sqrt{\hat{I}} . \]
\end{pro}
Let $\hat{Y}=\hat{a}(x,y) \partial / \partial x + \hat{b}(x,y) \partial / \partial y$ be a formal vector field.
We consider the set
\[ {\mathcal F} (\hat{Y}) = \{ \hat{g} \in {\mathbb C}[[x,y]] \ : \ \hat{Y}(\hat{g})=0 \}  \]
of first integrals of $\hat{Y}$. We say that $\hat{f} \in {\mathcal F}(\hat{Y})$ is primitive
if $\sqrt[k]{\hat{f}}$ does not belong to ${\mathbb C}[[x,y]]$ for $k>1$.
If ${\mathcal F}(\hat{Y}) \neq {\mathbb C}$ there exists a primitive formal first integral $\hat{f}$; moreover we have
\[ {\mathcal F}(\hat{Y}) = {\mathbb C}[[z]] \circ \hat{f}  \]
(Mattei-Moussu \cite{MaMo:Aen}), and the primitive first integral
can be chosen in ${\mathbb C}\{ x,y \}$ if $\hat{Y}$ is a germ of
holomorphic vector field.

We can give an alternative characterization for the first integrals of the infinitesimal generator
of a unipotent diffeomorphism.
\begin{lem}
\label{lem:expdif}
  Let $\sigma \in \diffh{u}{n}$ and $\hat{f} \in {\mathbb C}[[x_{1}, \hdots, x_{n}]]$. Then
\[ (\log \sigma) (\hat{f}) =0 \Leftrightarrow \hat{f} \circ \sigma = \hat{f}. \]
\end{lem}
\begin{proof}
We have
\[ \hat{f} \circ {\rm exp}(\log \sigma) = \hat{f} +  (\log \sigma) (\hat{f}) +
\frac{(\log \sigma)^{2}(\hat{f})}{2!} + \hdots  \]
Therefore $(\log \sigma) (\hat{f}) =0 $ implies $\hat{f} \circ \sigma = \hat{f}$.

Suppose $\hat{f} \circ \sigma = \hat{f}$. Denote by $\Theta$ the operator $\sigma - Id$.
We have $\Theta(\hat{f})=0$ by hypothesis and then $\Theta^{j}(\hat{f}) =0$ for any $j \in {\mathbb N}$.
We obtain $(\log \sigma) (\hat{f}) =0$ by equation (\ref{equ:log}).
\end{proof}
\section{Formal conjugacy}
Throughout this paper we work with germs of diffeomorphisms in $\cn{2}$ of the form
\[ \varphi_{\Delta,w}(x,y) = (x + y(y-x) \Delta(x,y), y + y(y-x)w(x,y))  \]
where $w,\Delta \in {\mathbb C} \{x,y\}$ and $w(0,0) \neq 0 = \Delta(0,0)$.
In this section we describe a geometrical condition for the formal class of $\varphi_{\Delta,w}$
not to be embeddable.

Next, we describe the structure of $\log \varphi_{\Delta,w}$.
\begin{defi}
\label{def:max}
Let ${\mathfrak m}_{x,y}$ and $\hat{\mathfrak m}_{x,y}$ be the
maximal ideals of the rings ${\mathbb C}\{x,y\}$ and ${\mathbb C}[[x,y]]$ respectively.
\end{defi}
\begin{lem}
\label{lem:strlog}
The formal vector field $\log \varphi_{\Delta,w}$ is of the form
\[ \log \varphi_{\Delta,w} = y(y-x) \left({
w(0,0) \frac{\partial}{\partial{y}} + h.o.t.
}\right)  \]
where $h.o.t$ stands for a formal vector field whose coefficients belong to $\hat{\mathfrak m}_{x,y}$.
\end{lem}
\begin{proof}
We denote by $\Theta$ the operator $\varphi_{\Delta, w} - Id$. Let $I$ be the ideal $(y(y-x))$ of
${\mathbb C}[[x,y]]$. We denote $X=\log \varphi_{\Delta,w}$. It is clear that
$\Theta ({\mathbb C}[[x,y]])$ is contained in $I$.
Thus $X({\mathbb C}[[x,y]])$ is contained in $I$ by equation (\ref{equ:log}).
Since we have $X= X(x) \partial / \partial x + X(y) \partial / \partial y$ and
$I \subset \hat{\mathfrak m}_{x,y}^{2}$ we obtain $X (I) \subset I \hat{\mathfrak m}_{x,y}$.
Given $g \in I$ we have
\[ \Theta (g) = X \left({ g + X \left({  \sum_{j=2}^{\infty} \frac{X^{j-2}(g)}{j!} }\right) }\right)
\in X(I +  X({\mathbb C}[[x,y]]) ) = X(I)  . \]
Therefore we get $\Theta(I) \subset X(I) \subset I \hat{\mathfrak m}_{x,y}$. The equation (\ref{equ:log})
implies
\[ X(x) = y(y-x) \Delta(x,y) +
\Theta \left({ \Theta \left({ \sum_{j=2}^{\infty} {(-1)}^{j+1} \frac{\Theta^{j-2} (x)}{j} }\right) }\right) \]
and then
\[ X(x) \in (y(y-x) \Delta(0,0) + I \hat{\mathfrak m}_{x,y}) + \Theta (I) =
y(y-x) \Delta(0,0) + I \hat{\mathfrak m}_{x,y} . \]
Analogously we obtain $X(y) -y(y-x) w(0,0) \in I \hat{\mathfrak m}_{x,y}$.
The lemma is a consequence of the equation
$X= X(x) \partial / \partial x + X(y) \partial / \partial y$.
\end{proof}
We denote the formal vector field $\log \varphi_{\Delta,w}/(y(y-x))$ by $L_{\Delta,w}$.
\begin{lem}
For any $\varphi_{\Delta,w}$ and $\hat{g} \in {\mathbb C}[[x]]$ there exists a unique
$\hat{f}$ in ${\mathbb C}[[x,y]]$ such that $(\log \varphi_{\Delta,w}) (\hat{f})=0$ and $\hat{f}(x,0)=\hat{g}(x)$.
\end{lem}
\begin{proof}
By lemma \ref{lem:strlog} we have that $L_{\Delta,w}(y)$ is a unit. Then
\[ (\log \varphi_{\Delta,w}) (\hat{f})=0 \Leftrightarrow \frac{\partial \hat{f}}{\partial{y}} =
- \frac{L_{\Delta,w}(x)}{L_{\Delta,w}(y)} \frac{\partial \hat{f}}{\partial{x}} . \]
As a consequence there is a unique formal solution of the previous equation fulfilling the
initial condition $\hat{f}(x,0)=\hat{g}(x)$.
\end{proof}
We want to introduce the formal invariants of $\varphi_{\Delta,w}$. The first formal invariant is
the fixed points set $Fix(\varphi_{\Delta,w})$.
\begin{pro}
\label{pro:invfix}
  Let $\tau_{1}, \tau_{2} \in \diff{}{n}$ and $\hat{\sigma} \in \diffh{}{n}$ such that
$\hat{\sigma} \circ \tau_{1} = \tau_{2} \circ \hat{\sigma}$. Then
we have $\hat{\sigma}(Fix (\tau_{1}))= Fix (\tau_{2})$.
\end{pro}
An equivalent statement is the following: Let $\hat{I}_{j} = I(Fix (\tau_{j}))$
for $j=1,2$. Then we have $\hat{I}_{2} \circ \hat{\sigma} = \hat{I}_{1}$.
\begin{proof}
  Let $\hat{\gamma} \in {\mathbb C}[[t]]^{n} \cap Z(\hat{I}_{1})$. We have
$\tau_{1} \circ \hat{\gamma}(t) = \hat{\gamma}(t)$; we obtain
\[ \hat{\sigma} \circ \tau_{1}(\hat{\gamma}(t)) = \tau_{2} \circ \hat{\sigma}(\hat{\gamma}(t)) \Rightarrow
 \tau_{2} \circ \hat{\sigma}(\hat{\gamma}(t)) =  \hat{\sigma} (\hat{\gamma}(t)) . \]
Hence $\hat{\sigma} \circ \hat{\gamma}(t)$ belongs to $Z(\hat{I}_{2})$ and then
$\hat{\sigma}(Z(\hat{I}_{1})) \subset Z(\hat{I}_{2})$. By the analogous argument applied to
$\hat{\sigma}^{(-1)}$ we obtain $Z(\hat{I}_{2}) \subset \hat{\sigma}(Z(\hat{I}_{1}))$ and then
$\hat{\sigma}(Z(\hat{I}_{1})) = Z(\hat{I}_{2})$. This is equivalent to
$I(Z(\hat{I}_{2})) \circ \hat{\sigma} = I(Z(\hat{I}_{1}))$. Since $\hat{I}_{1}$ and $\hat{I}_{2}$
are radical ideals then $\hat{I}_{2} \circ \hat{\sigma} = \hat{I}_{1}$ is a consequence of the formal theorem
of zeros.
\end{proof}
\begin{rem}
It is not required to use the formal theorem of zeros to prove the previous proposition but this proof
makes clear that the image by $\hat{\sigma}$ of a parametrization $\hat{\gamma}(t)$ of a formal curve
contained in $Fix (\tau_{1})$ is a parametrization $\hat{\sigma} \circ \hat{\gamma}(t)$ of a formal
curve contained in $Fix (\tau_{2})$.
\end{rem}
\begin{defi}
Let $a \in {\mathbb C}[[x]]$. We define
\[ \nu(a) = \sup \{ b \in {\mathbb N} \cup \{0\} : a \in (x^{b}) \}. \]
\end{defi}
\begin{lem}
\label{lem:trivial}
  Let $\varphi_{\Delta,w}$, $\tau \in \diff{}{2}$ and $\hat{\sigma} \in \diffh{}{2}$ such that
$\hat{\sigma} \circ \varphi_{\Delta,w} = \tau \circ \hat{\sigma}$. Then
\begin{enumerate}
\item $Fix (\tau)$ is an analytic set.
\item $Fix (\tau)$ has two irreducible components $f_{1}=0$ and $f_{2}=0$, both of them are smooth curves.
\item $\hat{\sigma}_{*} (\log \varphi_{\Delta,w})=\log \tau$.
\item $j^{0}(\log \tau/(f_{1}f_{2}))$ is transversal to both $f_{1}=0$ and $f_{2}=0$.
\item Let $\hat{f}$ be a primitive element of  ${\mathcal F}(\log \varphi_{\Delta,w})$.  Then
$\hat{f} \circ \hat{\sigma}^{(-1)}$ is a primitive element of ${\mathcal F}(\log \tau)$.
\item $\hat{f} \circ \hat{\sigma}^{(-1)}_{|f_{1}=0}$ and
$\hat{f} \circ \hat{\sigma}^{(-1)}_{|f_{2}=0}$ are ``injective". In other words, if $\hat{\gamma}(t)$
is a minimal parametrization of $f_{j}=0$ we have $\nu(\hat{f} \circ \hat{\sigma}^{(-1)} \circ \hat{\gamma})=1$.
\end{enumerate}
\end{lem}
\begin{proof}
Condition (1) is obvious.  Conditions (3) and (5) can be deduced of the uniqueness
of the infinitesimal generator.

Denote $\hat{\sigma}(x,y) = (\hat{\sigma}_{1}(x,y), \hat{\sigma}_{2}(x,y))$
and $\partial \hat{\sigma}/\partial z= (\partial \hat{\sigma}_{1}/\partial z, \partial \hat{\sigma}_{2}/ \partial z)$
for $z=x$ or $z=y$.
Since by prop. \ref{pro:invfix} we have $\hat{\sigma}(Fix(\varphi_{\Delta,w})) = Fix (\tau)$
then $\tau$ has two irreducible components $f_{1}=0$ and $f_{2}=0$
corresponding respectively to $\hat{\sigma}(y=0)$ and $\hat{\sigma}(y=x)$.
Moreover
\[ \hat{\sigma}(t,0)=(\hat{\sigma}_{1}(t,0), \hat{\sigma}_{2}(t,0)) \ {\rm and} \
\hat{\sigma}(t,t)=(\hat{\sigma}_{1}(t,t), \hat{\sigma}_{2}(t,t)) \]
are formal parametrizations of $f_{1}=0$ and $f_{2}=0$ respectively.
Since $j^{1} \hat{\sigma}$ is a linear isomorphism then the vectors
\[ \frac{\partial (\hat{\sigma}(t,0))}{\partial t} (0) = \frac{\partial \hat{\sigma}}{\partial x}(0,0)
\ {\rm and} \  \frac{\partial (\hat{\sigma}(t,t))}{\partial t} (0) =
\frac{\partial \hat{\sigma}}{\partial x}(0,0) + \frac{\partial \hat{\sigma}}{\partial y}(0,0) \]
are linear independent and then different than $(0,0)$.
An analytic minimal parametrization of $f_{1}=0$ is of the form
$\hat{\sigma}(\hat{a}(t),0)$ for some $\hat{a} \in {\mathbb C}[[t]]$ such that
$\nu(\hat{a})=1$. We obtain
\[  \frac{\partial (\hat{\sigma}(\hat{a}(t),0))}{\partial t} (0) =
\frac{\partial \hat{\sigma}}{\partial x}(0,0) \frac{\partial \hat{a}}{\partial t}(0) \neq (0,0) .\]
Thus $f_{1}=0$ is a smooth curve which is tangent at the origin to the line generated by
$(\partial \hat{\sigma}/\partial x)(0,0)$. Analogously the curve $f_{2}=0$
is smooth and tangent at the origin to the line generated by
$(\partial \hat{\sigma}/\partial x)(0,0)+ (\partial \hat{\sigma}/\partial y)(0,0)$.
Condition (3) and lemma \ref{lem:strlog} imply
\[ \log \tau = [ y (y-x)] \circ \hat{\sigma}^{-1}
\left({ w(0,0) \left({
\frac{\partial \hat{\sigma}_{1}}{\partial y}(0,0) \frac{\partial}{\partial x} +
\frac{\partial \hat{\sigma}_{2}}{\partial y}(0,0) \frac{\partial}{\partial y}
}\right) + h.o.t.  }\right) \]
where h.o.t stands for a formal vector field whose coefficients belong to $\hat{\mathfrak m}_{x,y}$.
The power series $[ y (y-x)] \circ \hat{\sigma}^{-1}$ is
of the form $f_{1} f_{2} \hat{u}$ where $\hat{u}$ is a unit of ${\mathbb C}[[x,y]]$.
We have
\[ j^{0} \left({ \frac{\log \tau}{f_{1}f_{2}} }\right) = \hat{u}(0,0) w(0,0)
\left({ \frac{\partial \hat{\sigma}_{1}}{\partial y}(0,0) \frac{\partial}{\partial x} +
\frac{\partial \hat{\sigma}_{2}}{\partial y}(0,0) \frac{\partial}{\partial y} }\right). \]
We obtain condition (4) since the vectors
\[ \frac{\partial \hat{\sigma}}{\partial y}(0,0), \  \frac{\partial \hat{\sigma}}{\partial x}(0,0)
\ {\rm and} \ \frac{\partial \hat{\sigma}}{\partial x}(0,0) +  \frac{\partial \hat{\sigma}}{\partial y}(0,0) \]
are pairwise non-collinear.

Condition (6) is equivalent to prove that
$\nu(\hat{f}(x,0))=\nu(\hat{f}(x,x))=1$ for every primitive
$\hat{f}$ in ${\mathcal F}(\log \varphi_{\Delta,w})$. We can suppose that
$\hat{f}(x,0)=x$ since then $\hat{f}$ is primitive and the set of
primitive elements of ${\mathcal F}(\log \varphi_{\Delta,w})$ is $\diffh{}{}
\circ \hat{f}$. The relation $L_{\Delta,w}(\hat{f}) = 0$ implies $j^{1} \hat{f}=x$. Therefore
$\nu(\hat{f}(x,0))=\nu(\hat{f}(x,x))=1$.
\end{proof}
Consider a couple $(S,g)$ where $S$ is a germ of analytic set and $g$ is a function on $S$.
We typically consider a couple $(\gamma, |Jac \  \varphi_{\Delta,w}|_{|\gamma})$ where $\gamma$ is a germ of
curve contained in $Fix (\varphi_{\Delta,w})$ and $|Jac \  \varphi_{\Delta,w}|$ is the determinant of the jacobian
matrix.
We denote ${\mathcal J}_{\Delta,w} =|Jac \  \varphi_{\Delta,w}|$ and
${\mathcal J}_{\tau} = |Jac \  \tau|$.
\begin{pro}
The couples
\[ (y=0,{({\mathcal J}_{\Delta,w})}_{|y=0}) \ \  {\rm and} \ \  (y=x,{({\mathcal J}_{\Delta,w})}_{|y=x}) \]
are formal invariants of $\varphi_{\Delta,w}$.
\end{pro}
\begin{proof}
Suppose $\hat{\sigma} \circ \varphi_{\Delta,w} = \tau \circ \hat{\sigma}$
for $\tau \in \diffh{}{2}$ and $\hat{\sigma} \in \diffh{}{2}$. We have
\[ ({\mathcal J}_{\hat{\sigma}} \circ \varphi_{\Delta,w}) {\mathcal J}_{\Delta,w} =
({\mathcal J}_{\tau} \circ \hat{\sigma}) {\mathcal J}_{\hat{\sigma}} \]
by the chain rule. Let ${\gamma}(t) \in (t {\mathbb C}\{t\})^{2}$ be a parametrization
of either $y=0$ or $y=x$. We have $\varphi_{\Delta,w} \circ \gamma(t)= \gamma(t)$; that implies
\[ {\mathcal J}_{\Delta,w} \circ \gamma(t)=  {\mathcal J}_{\tau} \circ (\hat{\sigma} \circ \gamma(t)) \]
as we wanted to prove.
\end{proof}
\begin{defi}
We define  $Tr_{\Delta,w}: (y=0) \to (y=x)$ as the unique formal mapping such that
$\hat{f}_{\Delta,w} \circ Tr_{\Delta,w} = \hat{f}_{\Delta,w}$
where $\hat{f}_{\Delta,w}$ is a primitive formal first integral of $\log \varphi_{\Delta,w}$.
By condition (6) in lemma \ref{lem:trivial} we have that $\hat{f}_{\Delta,w}(x,0)$ and $\hat{f}_{\Delta,w}(x,x)$
belong to $\diffh{}{}$. As a consequence
\[ Tr_{\Delta,w}(x,0) = ({(\hat{f}_{\Delta,w}(x,x))}^{(-1)} \circ \hat{f}_{\Delta,w}(x,0) ,
{(\hat{f}_{\Delta,w}(x,x))}^{(-1)} \circ \hat{f}_{\Delta,w}(x,0)) \]
is the expression of $Tr_{\Delta,w}$ in coordinates.
The mapping $Tr_{\Delta,w}$ does not depend on the choice of $\hat{f}_{\Delta,w}$. We call $Tr_{\Delta,w}$ the
{\it transport mapping}. If $\log \varphi_{\Delta,w}$ is a germ of vector field then $Tr_{\Delta,w}(x,0)$
is the only point in $y=x$ contained in the same trajectory of $L_{\Delta,w}$ than $(x,0)$.
\end{defi}
\begin{pro}
\label{pro:tramapfor}
  The transport mapping $Tr_{\Delta,w}$ associated to a diffeomorphism $\varphi_{\Delta,w}$ is a formal invariant.
\end{pro}
  Suppose $\hat{\sigma} \circ \varphi_{\Delta,w} = \tau \circ \hat{\sigma}$ for $\tau \in \diff{}{2}$ and
$\hat{\sigma} \in \diffh{}{2}$. In general the mapping $\tau$ is not necessarily of the form $\varphi_{\Delta',w'}$;
it is necessary to explain what we mean by the transport mapping of $\tau$.
By proposition \ref{pro:invfix} the formal curves
$\gamma_{1}= \hat{\sigma}(y=0)$ and $\gamma_{2}=\hat{\sigma}(y=x)$ are in fact analytic.
  We define $Tr_{\tau}: \gamma_{1} \to \gamma_{2}$ as the unique formal mapping such that
$\hat{g} \circ Tr_{\tau} = \hat{g}$ for every primitive
$\hat{g}$ in ${\mathcal F}(\log \tau)$. Then $Tr_{\tau}$ is well-defined by lemma \ref{lem:trivial}.
\begin{proof}
We keep the notations in the previous paragraph. We want to prove the equality
$(Tr_{\tau}  \circ \hat{\sigma} )(x,0) = (\hat{\sigma} \circ Tr_{\Delta,w}) (x,0)$.
We choose a primitive $\hat{f}_{\Delta,w} \in {\mathcal F}(\log \varphi_{\Delta,w})$;
the series $\hat{g} = \hat{f}_{\Delta,w} \circ \hat{\sigma}^{(-1)}$ is a primitive element of ${\mathcal F}(\log \tau)$
(lemma \ref{lem:trivial}).  Since $\hat{f}_{\Delta,w}(x,0) = \hat{f}_{\Delta,w}( Tr_{\Delta,w} (x,0))$
by definition of $Tr_{\Delta,w}$ then we obtain
\[ (\hat{f}_{\Delta,w} \circ \hat{\sigma}^{(-1)}) (\hat{\sigma} (x,0)) = (\hat{f}_{\Delta,w} \circ \hat{\sigma}^{(-1)})
(\hat{\sigma}(Tr_{\Delta,w} (x,0)))  . \]
The definition of $Tr_{\tau}$ implies $Tr_{\tau} (\hat{\sigma} (x,0)) = \hat{\sigma}(Tr_{\Delta,w} (x,0))$
as we wanted to prove.
\end{proof}
\begin{defi}
The formal class of a unipotent diffeomorphism $\tau$ is embeddable if
$\log \tau$ is formally conjugated
to a convergent germ of vector field.
\end{defi}
Next we introduce an obstruction to the embeddability of a formal class.
\begin{pro}
\label{pro:tracon}
  Suppose that there exist $X \in {\mathcal X}_{N} \cn{2}$ and $\hat{\sigma}$ in $\diffh{}{2}$
such that $\hat{\sigma} \circ \varphi_{\Delta,w} = {\rm exp}(X) \circ \hat{\sigma}$. Then
$Tr_{\Delta,w}$ is a convergent mapping.
\end{pro}
\begin{proof}
The transport mapping $Tr_{\Delta,w}$ maps $y=0$ to $y=x$, thus it satisfies
$Tr_{\Delta,w}(x,0) \equiv (\hat{a}(x), \hat{a}(x))$ for some $\hat{a} \in {\mathbb C}[[x]]$.
It suffices to prove that $\hat{a}$ belongs to ${\mathbb C} \{ x \}$.
The transport mapping is a formal invariant (prop. \ref{pro:tramapfor}), hence we have
$\hat{\sigma} \circ Tr_{\Delta,w}(x,0) = Tr_{{\rm exp}(X)} \circ \hat{\sigma}(x,0)$. The previous
equation implies
\[   \hat{a}(x) = (\hat{\sigma}(x,x) )^{-1} \circ Tr_{{\rm exp}(X)} \circ \hat{\sigma}(x,0)   . \]
Since $Tr_{\rm exp(X)}$ is convergent
then it suffices to prove that $\hat{\sigma}(x,0)$ and $\hat{\sigma}(x,x)$
belong to ${\mathbb C}\{ x \}^{2}$.

 We have ${\mathcal J}_{\Delta,w} = 1 + (2y-x) w(0,0) + h.o.t$. Therefore the function
 ${({\mathcal J}_{\Delta,w})}_{|y=0}$ is injective. Consider a convergent minimal parametrization
 $\eta(x)$ of $\hat{\sigma}(y=0)$; there exists $\hat{h} \in \diffh{}{}$ such that
 $\hat{\sigma}(x,0)= \eta \circ \hat{h}(x)$. Since
 ${\mathcal J}_{{\rm exp}(X)} \circ \hat{\sigma}(x,0) ={\mathcal J}_{\Delta,w}(x,0)$ then
 \[ \frac{\partial}{\partial{x}}
 {\left({
 {\mathcal J}_{{\rm exp}(X)} \circ \eta(x)
 }\right)}(0) = - \frac{w(0,0)}{\partial{\hat{h}}/{\partial{x}}(0)} \neq 0 . \]
As a consequence
\[ \hat{h} = ({\mathcal J}_{{\rm exp}(X)} \circ \eta(x) -1)^{(-1)} \circ ({\mathcal J}_{\Delta,w}(x,0)-1) \]
belongs to $\diff{}{}$. That implies $\hat{\sigma}(x,0) =\eta \circ \hat{h} \in {\mathbb C}\{x\}^{2}$.
The proof for $\hat{\sigma}(x,x)$ is analogous.
\end{proof}
\begin{rem}
In order to find a unipotent diffeomorphism whose formal class is non-embeddable
it suffices to exhibit $\varphi_{\Delta,w}$ such that $Tr_{\Delta,w}$ is divergent.
\end{rem}
\begin{rem}
We do not prove it in this paper but a diffeomorphism $\varphi_{\Delta,w}$ such that
$Tr_{\Delta,w}$ is an analytic mapping has embeddable formal class. In particular a diffeomorphism
$\varphi_{0,w} = (x, y + y(y-x)w(x,y))$ has embeddable formal class.
\end{rem}
\section{Polynomial families}
\label{sec:polfam}
Consider the family
\[ \varphi_{\lambda \Delta, w}= (x + \lambda y (y-x) \Delta(x,y), y + y(y-x)w(x,y)) \]
where $w(0,0) \neq 0 = \Delta(0,0)$ and $\lambda \in {\mathbb C}$.  It is affine on $\lambda$.
We denote by $\hat{f}_{\lambda}$ the only element of ${\mathcal F}(\log \varphi_{\lambda \Delta,w})$ such that
$\hat{f}_{\lambda}(x,0)=x$. The transport mapping $Tr_{\lambda \Delta,w}$ satisfies
\[ Tr_{\lambda \Delta, w}(x,0)   = ({(\hat{f}_{\lambda}(x,x))}^{(-1)} \circ \hat{f}_{\lambda}(x,0)  ,
{(\hat{f}_{\lambda}(x,x))}^{(-1)} \circ \hat{f}_{\lambda}(x,0) ) \]
and then $Tr_{\lambda \Delta, w}(x,0)     = ({(\hat{f}_{\lambda}(x,x))}^{(-1)}  , {(\hat{f}_{\lambda}(x,x))}^{(-1)})$.
As a consequence $Tr_{\lambda \Delta,w}$ is convergent if and only if
$\hat{f}_{\lambda}(x,x) \in {\mathbb C}\{x\}$.
\begin{lem}
\label{lem:logpol}
We have
\[ \frac{\log \varphi_{\lambda \Delta,w}}{y(y-x)} = \left({ \sum_{0 \leq k,l} a_{k,l}^{1}(\lambda) {x}^{k} {y}^{l}
}\right) \frac{\partial}{\partial{x}}
+ \left({
\sum_{0 \leq k,l} a_{k,l}^{2}(\lambda) {x}^{k} {y}^{l}
}\right) \frac{\partial}{\partial{y}}  \]
where $a_{k,l}^{j} \in {\mathbb C} [ \lambda ]$ for all $k,l \geq 0$ and $j \in \{1,2\}$.
\end{lem}
\begin{proof}
Denote $X_{\lambda} = \log \varphi_{\lambda \Delta,w}$. Denote by $\Theta_{\lambda}$ the operator
$\varphi_{\lambda \Delta,w} - Id$ for $\lambda \in {\mathbb C}$. Given $j \geq 0$ consider
an element $g$ of $(y(y-x)) \hat{\mathfrak m}_{x,y}^{j}$. The series $g$ belongs to
$\hat{\mathfrak m}_{x,y}^{j+2}$; we obtain
\[ X_{\lambda}(g) =X_{\lambda}(x) \frac{\partial g}{\partial x} +  X_{\lambda}(y) \frac{\partial g}{\partial y}
\in (y(y-x))  \hat{\mathfrak m}_{x,y}^{j+2-1} = (y(y-x))  \hat{\mathfrak m}_{x,y}^{j+1} \]
by lemma \ref{lem:strlog} for any $\lambda \in {\mathbb C}$. We iterate the argument to get
\begin{equation}
\label{equ:stair}
X_{\lambda}^{k}((y(y-x)) \hat{\mathfrak m}_{x,y}^{j}) \subset (y(y-x)) \hat{\mathfrak m}_{x,y}^{j+k}
\end{equation}
for all $j, k \geq 0$ and $\lambda \in {\mathbb C}$. Moreover, since
$\Theta_{\lambda}(g) = \sum_{k=1}^{\infty} X_{\lambda}^{k}(g)/j!$
the equation (\ref{equ:stair}) implies
$\Theta_{\lambda}((y(y-x)) \hat{\mathfrak m}_{x,y}^{j}) \subset (y(y-x)) \hat{\mathfrak m}_{x,y}^{j+1}$
for any $\lambda \in {\mathbb C}$. We obtain
\begin{equation}
\label{equ:xua2}
\Theta_{\lambda}^{k}((y(y-x)) \hat{\mathfrak m}_{x,y}^{j}) \subset (y(y-x)) \hat{\mathfrak m}_{x,y}^{j+k}
\end{equation}
for all $j \geq 0$, $k \geq 1$ and $\lambda \in {\mathbb C}$.
The series $\Theta_{\lambda}(g)$ belongs to $(y(y-x))$, thus the equation (\ref{equ:xua2})
implies $\Theta_{\lambda}^{j}(g) = \Theta_{\lambda}^{j-1}(\Theta_{\lambda}(g))
\in (y(y-x)) \hat{\mathfrak m}_{x,y}^{j-1}$
for all $g \in {\mathbb C}[[x,y]]$, $j \geq 1$ and $\lambda \in {\mathbb C}$.
Since $X_{\lambda}(g) = \sum_{j=1}^{\infty} (-1)^{j+1} \Theta_{\lambda}^{j}(g)/j$ then
\[ a_{k,l}^{1}(\lambda)  = \frac{1}{k! l!} \sum_{j=1}^{k+l+1} \frac{(-1)^{j+1}}{j}
\frac{\partial^{k+l} [\Theta_{\lambda}^{j}(x)/(y(y-x))]}{\partial x^{k} \partial y^{l}}(0,0) . \]
An analogous expression is obtained for $a_{k,l}^{2}$ for all $k,l \geq 0$. It suffices to prove
that $\Theta_{\lambda}({\mathbb C}[\lambda][[x,y]]) \subset {\mathbb C}[\lambda][[x,y]]$; then
$a_{k,l}^{1}$ and $a_{k,l}^{2}$ are finite sums of polynomials for all $k,l \geq 0$.

Given $h = \sum_{k,l \geq 0} h_{k,l}(\lambda) x^{k} y^{l} \in {\mathbb C}[\lambda][[x,y]]$ we have
\[ \Theta_{\lambda}(h) = \sum_{k,l \geq 0} h_{k,l}(\lambda) \Theta_{\lambda}(x^{k} y^{l}) \]
and then
\begin{equation}
\label{equ:polyn}
\Theta_{\lambda}(h) =
\sum_{k,l \geq 0} h_{k,l}(\lambda) [   (x+y(y-x) \lambda \Delta )^{k} (y+y(y-x)w )^{l} -  x^{k} y^{l}] .
\end{equation}
Since $\Theta_{\lambda}(x^{k} y^{l}) \in \hat{\mathfrak m}_{x,y}^{k+l}$ for all $k,l \geq 0$ and
$\lambda \in {\mathbb C}$ then the series (\ref{equ:polyn}) converges in the Krull topology to an element of
${\mathbb C}[\lambda][[x,y]]$.
\end{proof}
\begin{pro}
Let $\hat{f}_{\lambda}$ be the unique formal first integral of $\log \varphi_{\lambda \Delta,w}$ such
that $\hat{f}_{\lambda}(x,0)=x$. Then $\hat{f}_{\lambda}$ can be expressed in the form
\[ \hat{f}_{\lambda} = x + y \sum_{j+k \geq 1} f_{j,k}(\lambda) {x}^{j} {y}^{k} \]
where $f_{j,k} \in {\mathbb C}[\lambda]$ and $\deg f_{j,k} \leq j+k$ for any $j+k \geq 1$.
\end{pro}
\begin{proof}
Let $\tilde{\mathfrak m}$ be the ideal $(x,y)$ of the ring ${\mathbb C}[\lambda][[x,y]]$.
The property $(\log \varphi_{\lambda \Delta,w})  (\hat{f}_{\lambda})=0$ is equivalent to
$L_{\lambda \Delta, w} (\hat{f}_{\lambda})=0$.
We denote
\[ L_{\lambda \Delta, w}= a(\lambda, x,y) \partial / \partial x + b(\lambda, x,y) \partial / \partial y. \]
Then $a$ and $b$ belong to ${\mathbb C}[\lambda][[x,y]]$ by lemma \ref{lem:logpol}.
Moreover, we have $a \in \tilde{\mathfrak m}$ and $b(\lambda,0,0) \equiv w(0,0)$ by
lemma \ref{lem:strlog}.
The series $b$ is of the form $w(0,0) (1 - b_{0})$ where $b_{0} \in \tilde{\mathfrak m}$. Hence we obtain
\begin{equation}
\label{equ:equsim}
a  \frac{\partial \hat{f}_{\lambda}}{\partial x} +
b  \frac{\partial \hat{f}_{\lambda}}{\partial y} = 0 \Rightarrow
\frac{\partial \hat{f}_{\lambda}}{\partial y} = - a w(0,0)^{-1} (\sum_{j=0}^{\infty} b_{0}^{j})
\frac{\partial \hat{f}_{\lambda}}{\partial x} .
\end{equation}
Since $b_{0}^{j} \in \tilde{\mathfrak m}^{j}$ for any $j \geq 0$ then
$\sum_{j=0}^{\infty} b_{0}^{j}$ converges in the Krull topology to an element of ${\mathbb C}[\lambda][[x,y]]$.
Denote $c= - a w(0,0)^{-1} (\sum_{j=0}^{\infty} b_{0}^{j})$.
Then $c$ belongs to
${\mathbb C}[\lambda][[x,y]] \cap \tilde{\mathfrak m}$. The series
$c$ can be expressed in the form
$\sum_{j=0}^{\infty} c_{j}(\lambda,x) y^{j}$ where $c_{j} \in {\mathbb C}[\lambda][[x]]$
for any $j \geq 0$. Moreover $c_{0}$ belongs to the ideal $(x)$ of ${\mathbb C}[\lambda][[x]]$.
Denote $\hat{f}_{\lambda}= \sum_{j=0}^{\infty} f_{j}(\lambda,x) y^{j}$.
We have $f_{0} \equiv x$ by choice. We obtain
\begin{equation}
\label{equ:coefl}
\sum_{j=1}^{\infty} j f_{j}(\lambda,x) y^{j-1} = (\sum_{j=0}^{\infty} c_{j}(\lambda,x) y^{j})
\left({ \sum_{j=0}^{\infty} \frac{f_{j}(\lambda,x)}{\partial x} y^{j} }\right)
\end{equation}
by developing equation (\ref{equ:equsim}).
By comparing the coefficients of $y^{0}$ in both sides of equation (\ref{equ:coefl}) we obtain
$f_{1}=c_{0} (\partial f_{0}/\partial x)=c_{0} \in (x) \subset {\mathbb C}[\lambda][[x]]$.
Suppose that $f_{0}, \hdots, f_{k}$ belong to ${\mathbb C}[\lambda][[x]]$ for some $k \geq 0$. Since
\[ (k+1) f_{k+1} = \sum_{j=0}^{k} c_{j} (\partial f_{k-j}/ \partial x) \]
then $f_{k+1}$ belongs to ${\mathbb C}[\lambda][[x]]$. We deduce that $f_{j}$ belongs to
${\mathbb C}[\lambda][[x]]$ for any $j \geq 0$ by induction. Therefore
$\hat{f}_{\lambda}$ belongs to ${\mathbb C}[\lambda][[x,y]]$ and it can be expressed in the
form $ x + y \sum_{j+k \geq 0} f_{j,k}(\lambda) {x}^{j} {y}^{k}$ where
$f_{j,k} \in {\mathbb C}[\lambda]$ for all $j,k \geq 0$. Moreover $f_{0,0}$ is identically zero
since $f_{1}(\lambda,0) \equiv 0$.

  Consider
\[ \tau_{\Delta,w,\lambda}(x,y)= \left({ \frac{x}{\lambda}, \frac{y}{\lambda} }\right)
\circ \varphi_{\Delta/\lambda, w} \circ (\lambda x, \lambda y) . \]
We have
\[ \tau_{\Delta,w,\lambda}(x,y) = (x + y(y-x) \Delta(\lambda x, \lambda y), y + \lambda y (y-x) w(\lambda x, \lambda y)) . \]
We can proceed as in lemma \ref{lem:strlog} to prove
\[ \log \tau_{\Delta,w,\lambda}(x,y) = \lambda y (y-x) \left({
w(0,0) \frac{\partial}{\partial{y}} + h.o.t. }\right) . \]
There is an analogue of lemma \ref{lem:logpol} for $\log \tau_{\Delta,w,\lambda}/(\lambda y (y-x))$.
Again such an expression can be used to prove that the unique first integral
\[ \hat{g}_{\lambda} = x + y \sum_{j+k \geq 1} g_{j,k}(\lambda) {x}^{j} {y}^{k} \]
of $\log \tau_{\Delta,w,\lambda}$ such that $\hat{g}_{\lambda}(x,0)=x$ satisfies
$g_{j,k} \in {\mathbb C}[\lambda]$ for any
$j+k \geq 1$. The relation between $\varphi_{\lambda \Delta,w}$ and $\tau_{\Delta,w,\lambda}$ implies
\[ \hat{f}_{1/\lambda}(\lambda x, \lambda y) = \lambda \hat{g}_{\lambda}(x,y) . \]
We obtain $\lambda {f}_{j,k}(1/\lambda) {\lambda}^{j+k} = \lambda g_{j,k}(\lambda)$ for any $j+k \geq 1$.
Since ${f}_{j,k}$ and ${g}_{j,k}$ are polynomials we deduce that ${f}_{j,k}$
is a polynomial of degree at most $j+k$ for any $j+k \geq 1$.
\end{proof}
  We have $\hat{f}_{\lambda}(x,x) = x + \sum_{j+k \geq 1} f_{j,k}(\lambda) {x}^{j+k+1}$.
Next result is crucial.
\begin{pro}[\cite{PM2,PM4}]
\label{pro:PM}
Let $\hat{P} = \sum_{j \geq 0} P_{j}(\lambda) {x}^{j}$ where $P_{j} \in {\mathbb C}[\lambda]$ and
$\deg P_{j} \leq Aj+B$ for some $A, B \in {\mathbb R}$ and any $j \in {\mathbb N}$. Then either
$\hat{P}(\lambda,x)$ is convergent in a neighborhood of $x=0$ or
$\hat{P}(\lambda) \in {\mathbb C}[[x]] \setminus {\mathbb C}\{x\}$ for any $\lambda \in {\mathbb C}$
outside a polar set.
\end{pro}
The series $\hat{P}(\lambda,x)$ is convergent in a neighborhood of $x=0$ if there exists a neighborhood
$V$ of $x=0$ in ${\mathbb C}^{2}$ and a function $P$ holomorphic in $V$ such that $\hat{P}(\lambda,x)$
is the power series development of $P$. This property implies
$\hat{P}(\lambda,x) \in {\mathbb C}\{x\}$ for any $\lambda \in {\mathbb C}$.
The reciprocal is false if we drop the condition on the linear growth of $\deg P_{j}$.
%
%and $\hat{P}(\lambda,x)$ depends holomorphically not only on $x$ but also on $\lambda$.

 A polar set (see \cite{ran}) has measure zero as well as Haussdorff dimension zero.
Moreover, it is totally disconnected.
\begin{cor}
\label{cor:polar}
  Fix $\Delta \in {\mathfrak m}_{x,y}$ and $w \in  {\mathbb C}\{x,y\}$ with $w(0,0) \neq 0$. Either
$(\lambda, x) \mapsto Tr_{\lambda \Delta, w}(x)$ is convergent in a neighborhood of $x=0$ or
$x \mapsto Tr_{\lambda \Delta, w}(x)$ is divergent for any $\lambda \in {\mathbb C}$ outside a polar set.
\end{cor}
\begin{pro}
\label{pro:proequ}
Let  $w \in  {\mathbb C}\{x,y\} \setminus {\mathfrak m}_{x,y}$ and $\Delta \in {\mathfrak m}_{x,y}$
(see def. \ref{def:max}). Suppose that the formal class of $\varphi_{\lambda \Delta,w}$ is
embeddable for any $\lambda \in {\mathbb C}$. Then the equation
\[ \hat{\epsilon} - \hat{\epsilon} \circ \varphi_{0,w} = y(y-x) \Delta(x,y)
\ \ \ {\rm (homological \ equation)} \]
has a solution $\hat{\epsilon}_{\Delta} \in {\mathbb C}[[x,y]]$ such that
$\hat{\epsilon}_{\Delta}(x,x) - \hat{\epsilon}_{\Delta}(x,0) \in {\mathbb C}\{x \}$.
\end{pro}
\begin{proof}
Let $\hat{f}_{\lambda}$ be the first integral of $\log \varphi_{\lambda \Delta,w}$ such that
$\hat{f}_{\lambda}(x,0)=x$. It satisfies $\hat{f}_{\lambda} \circ \varphi_{\lambda \Delta,w} = \hat{f}_{\lambda}$
for any $\lambda \in {\mathbb C}$ by lemma \ref{lem:expdif}. We obtain
\[ {\frac{\partial(\hat{f}_{\lambda} \circ \varphi_{\lambda \Delta,w})}{\partial{\lambda}}}_{|\lambda=0}
= {\frac{\partial \hat{f}_{\lambda}}{\partial{\lambda}}}_{|\lambda=0} . \]
Since $\hat{f}_{0}=x$ then
\[ y(y-x) \Delta(x,y)  + \left({
{\frac{\partial \hat{f}_{\lambda}}{\partial{\lambda}}}_{|\lambda=0} }\right)
\circ \varphi_{0,w} = {\frac{\partial \hat{f}_{\lambda}}{\partial{\lambda}}}_{|\lambda=0} . \]
We define $\hat{\epsilon}_{\Delta}(x,y) = (\partial \hat{f}_{\lambda}/ \partial{\lambda})(0,x,y)$.
We have $\hat{\epsilon}_{\Delta}(x,0)=0$ by definition of $\hat{f}_{\lambda}$.
The corollary \ref{cor:polar} and proposition \ref{pro:tracon} imply that
$(\lambda, x) \mapsto \hat{f}_{\lambda}(x,x)$ is convergent in a neighborhood of $(\lambda,x)=(0,0)$.
Therefore $\hat{\epsilon}_{\Delta}(x,x) \in {\mathbb C} \{ x \}$ and then
$\hat{\epsilon}_{\Delta}(x,x) - \hat{\epsilon}_{\Delta}(x,0) \in {\mathbb C}\{x \}$.
\end{proof}
Note that in this section we related the existence of a diffeomorphism $\varphi_{\Delta,w}$
with $\Delta \neq 0$ and no embeddable formal class with the properties of $\varphi_{0,w}$
(which by the way has embeddable formal class).

The rest of the paper is devoted to prove the existence of a homological equation such that
$\hat{\epsilon}(x,x) - \hat{\epsilon}(x,0)$ diverges for any formal solution $\hat{\epsilon}(x,y)$.
%
%
%The diffeomorphism $\varphi_{0,w}$ satisfies $x \circ \varphi_{0,w} = x$, moreover it has a convergent
%normal form. That is a significant progress since in order to find a $\varphi_{\Delta,w}$ without a
%convergent normal form we only have to deal with simpler diffeomorphisms.
\section{The homological equation}
The main result in this section is proving that the homological equation
can be replaced by a differential equation.
\begin{lem}
\label{lem:pary}
$\log \varphi_{0,w}$ is of the form $y(y-x) (w(0,0) + h.o.t.) \partial / \partial{y}$.
\end{lem}
\begin{proof}
  By lemma \ref{lem:strlog} $\log \varphi_{0,w}$ is of the form
\[ y(y-x) (a(x,y) \partial / \partial x + b(x,y) \partial / \partial y) \]
where $a,b \in {\mathbb C}[[x,y]]$, $a(0,0)=0$ and $b(0,0)=w(0,0)$. It suffices to prove
that $a \equiv 0$ or equivalently $(\log \varphi_{0,w})(x) = 0$.
Since $x \circ \varphi_{0,w} =x$ then we have $(\log \varphi_{0,w})(x) = 0$ by
lemma \ref{lem:expdif}.
\end{proof}
\begin{lem}
\label{lem:funwd1}
The equation $\hat{\epsilon}  - \hat{\epsilon}  \circ \varphi_{0,w} = y(y-x) \Delta(x,y)$
has a solution $\hat{\epsilon} = \hat{\epsilon}_{\Delta}\in {\mathbb C}[[x,y]]$ for any
$\Delta \in {\mathbb C}[[x,y]]$.
\end{lem}
\begin{proof}
We define $\tilde{\nu}(A)= \sup \{ j \in {\mathbb N} \cup \{0\} : A \in \hat{\mathfrak m}_{x,y}^{j} \}$
for $A \in {\mathbb C}[[x,y]]$. We define $\Delta_{0}=\Delta$. Consider the equation
\begin{equation}
\label{equ:plug}
(\log \varphi_{0,w}) (\tilde{\epsilon}_{0}) = -y (y-x) \Delta_{0} .
\end{equation}
Since $L_{0,w}$ is non-singular at $(0,0)$ there exists a solution
$\tilde{\epsilon}_{0} \in {\mathbb C}[[ x,y ]]$ such that
$\tilde{\nu}(\tilde{\epsilon}_{0}) \geq \tilde{\nu}(\Delta_{0}) + 1$.
We consider
\begin{equation}
\label{equ:red1}
 (\tilde{\epsilon}_{0} + \epsilon_{1}) - (\tilde{\epsilon}_{0} + \epsilon_{1}) \circ \varphi_{0,w}=
y(y-x) \Delta_{0} .
\end{equation}
A solution $\epsilon_{1} \in {\mathbb C}[[x,y]]$ of equation (\ref{equ:red1}) provides a
solution of the original equation. By plugging the equation (\ref{equ:plug}) in the development of
the exponential $\tilde{\epsilon}_{0} \circ \varphi_{0,w} = \tilde{\epsilon}_{0}
+ \sum_{j=1}^{\infty} (\log \varphi_{0,w})^{j}(\tilde{\epsilon}_{0}) / j!$ we obtain
\[ \tilde{\epsilon}_{0} \circ \varphi_{0,w} = \tilde{\epsilon}_{0} - y (y-x) \Delta_{0}
+ \sum_{j=2}^{\infty} \frac{(\log \varphi_{0,w})^{j-1}(-y(y-x) \Delta_{0})}{j!} . \]
As a consequence the equation (\ref{equ:red1}) is equivalent to
\[ \epsilon_{1} - \epsilon_{1} \circ \varphi_{0,w} = \sum_{k \geq 2} \frac{1}{k!}
(\log \varphi_{0,w})^{k-1} (-y(y-x) \Delta_{0}(x,y)) . \]
We denote the term in the right-hand side by $y(y-x) \Delta_{1}$. The series $\Delta_{1}$ satisfies
$\tilde{\nu}(\Delta_{1}) \geq \tilde{\nu}(\Delta_{0})+1$. We have
$\tilde{\epsilon}_{0} - \tilde{\epsilon}_{0} \circ \varphi_{0,w} = y(y-x) (\Delta - \Delta_{1})$
by construction. We proceed by induction. Given
$\Delta_{j} \in  {\mathbb C}[[x,y]]$ there
exists $\tilde{\epsilon}_{j} \in {\mathbb C}[[x,y]]$ such that
$(\log \varphi_{0,w}) (\tilde{\epsilon}_{j}) = -y (y-x) \Delta_{j}$
and $\tilde{\nu}(\tilde{\epsilon}_{j}) \geq \tilde{\nu}(\Delta_{j})+1$. As previously we define
\[ y (y-x) \Delta_{j+1} = \sum_{k \geq 2} \frac{1}{k!}
(\log \varphi_{0,w})^{k-1} (-y(y-x) \Delta_{j}(x,y)) . \]
We obtain  $\tilde{\nu}(\Delta_{j+1}) \geq \tilde{\nu}(\Delta_{j})+1$ for any $j \geq 0$.
The construction implies that
$\tilde{\epsilon}_{j} - \tilde{\epsilon}_{j} \circ \varphi_{0,w} = y(y-x) (\Delta_{j} - \Delta_{j+1})$
and then
\begin{equation}
\label{equ:finlevl}
 (\tilde{\epsilon}_{0} + \hdots + \tilde{\epsilon}_{j}) -
(\tilde{\epsilon}_{0} + \hdots + \tilde{\epsilon}_{j}) \circ \varphi_{0,w} =
y(y-x) (\Delta - \Delta_{j+1})
\end{equation}
for any $j \in {\mathbb N} \cup \{0\}$. Since $\tilde{\nu}(\Delta_{j+1}) \geq \tilde{\nu} (\Delta_{j}) +1$
for $j \geq 0$ then $\tilde{\nu}(\Delta_{j}) \geq j$ and $\tilde{\nu}(\tilde{\epsilon}_{j}) \geq j+1$
for any $j \geq 0$. The series $\Delta - \Delta_{j+1}$ converges to $\Delta$ in the Krull topology
when $j \to \infty$. Analogously $\sum_{k=0}^{j} \tilde{\epsilon}_{j}$
converges in the Krull topology to some $\hat{\epsilon}_{\Delta} \in {\mathbb C}[[x,y]]$
when $j \to \infty$. By taking limits in equation (\ref{equ:finlevl}) when $j \to \infty$
we obtain $\hat{\epsilon}_{\Delta} - \hat{\epsilon}_{\Delta} \circ \varphi_{0,w} = y(y-x) \Delta$.
\end{proof}
\begin{lem}
\label{lem:funwd2}
Fix $\Delta \in {\mathbb C}[[x,y]]$.
The series $\hat{\epsilon}_{\Delta}(x,x) - \hat{\epsilon}_{\Delta}(x,0)$ does not depend on the
solution $\hat{\epsilon}_{\Delta} \in {\mathbb C}[[x,y]]$ of
$\hat{\epsilon}  - \hat{\epsilon}  \circ \varphi_{0,w} = y(y-x) \Delta(x,y)$.
\end{lem}
\begin{proof}
  It suffices to prove $\hat{\epsilon}(x,x) - \hat{\epsilon}(x,0) =0$ for any solution
$\hat{\epsilon}$ in ${\mathbb C}[[x,y]]$ of
$\hat{\epsilon}  - \hat{\epsilon}  \circ \varphi_{0,w} = 0$. The series $\hat{\epsilon}$
belongs to ${\mathcal F}(\log \varphi_{0,w})$ by lemma \ref{lem:expdif}. Moreover, lemma \ref{lem:pary}
implies $\partial{\hat{\epsilon}}/\partial{y} =0$. Hence $\hat{\epsilon}$ belongs to
${\mathbb C}[[x]]$ and clearly $\hat{\epsilon}(x,x) - \hat{\epsilon}(x,0) =0$.
\end{proof}
Given $\Delta \in {\mathbb C}[[x,y]]$ and a solution
$\hat{\epsilon}_{\Delta} \in {\mathbb C}[[x,y]]$ of the equation
\[ \hat{\epsilon}  - \hat{\epsilon}  \circ \varphi_{0,w} = y(y-x) \Delta(x,y) \]
we define $S_{w}(\Delta) = \hat{\epsilon}_{\Delta}(x,x) - \hat{\epsilon}_{\Delta}(x,0)$.
The lemmas \ref{lem:funwd1} and \ref{lem:funwd2}
imply that $S_{w}:{\mathbb C}[[x,y]] \to {\mathbb C}[[x]]$ is a well-defined linear functional.
Proposition \ref{pro:proequ} implies that if $\varphi_{\Delta,w}$ has embeddable formal class for
any $\Delta \in {\mathfrak m}_{x,y}$ then $S_{w} ({\mathfrak m}_{x,y}) \subset {\mathbb C}\{x\}$.
\begin{lem}
\label{lem:izs}
  Let $w \in {\mathbb C}\{x,y\} \setminus {\mathfrak m}_{x,y}$. Then we have
\[ S_{w} \left({ \frac{\log \varphi_{0,w}}{y(y-x)} [y(y-x)\Delta(x,y)] }\right) =0 \]
for any $\Delta \in {\mathbb C}[[x,y]]$.
\end{lem}
\begin{proof}
Let $\hat{\epsilon}_{0} \in {\mathbb C}[[x,y]]$ be a solution of
$\hat{\epsilon}  - \hat{\epsilon}  \circ \varphi_{0,w} = y(y-x) \Delta(x,y)$.
Let $t \in {\mathbb C}$.
By pre-composing the previous equation with ${\rm exp}(t \ln \varphi_{0,w})$ we obtain
\[ \hat{\epsilon}_{0} \circ {\rm exp}(t \ln \varphi_{0,w}) -
\hat{\epsilon}_{0} \circ \varphi_{0,w} \circ {\rm exp}(t \ln \varphi_{0,w})
=  [y(y-x) \Delta] \circ {\rm exp}(t \ln \varphi_{0,w}). \]
Denote $\hat{\epsilon}_{t} = \hat{\epsilon}_{0} \circ {\rm exp}(t \ln \varphi_{0,w})$.
The formal diffeomorphisms $\varphi_{0,w}$ and ${\rm exp}(t \ln \varphi_{0,w})$ commute
for any $t \in {\mathbb C}$ (see equation (\ref{equ:comm}) in section \ref{sec:basic}).
Thus $\hat{\epsilon}_{t}$ satisfies the equation
\[ \hat{\epsilon}_{t}  - \hat{\epsilon}_{t}  \circ \varphi_{0,w} = [y(y-x) \Delta(x,y)] \circ
{\rm exp}(t \ln \varphi_{0,w}) \]
for any $t \in {\mathbb C}$. Moreover, we have
\[ \hat{\epsilon}_{0} \circ {\rm exp}(t \ln \varphi_{0,w})(x,x) -
\hat{\epsilon}_{0} \circ {\rm exp}(t \ln \varphi_{0,w})(x,0) = \hat{\epsilon}_{0}(x,x) - \hat{\epsilon}_{0}(x,0) . \]
That implies
\[ S_{w} \left({
\frac{[y(y-x)\Delta(x,y)] \circ {\rm exp}(t \ln \varphi_{0,w}) - [y(y-x)\Delta(x,y)]}{t y (y-x)}
}\right) =0 \]
for any $t \in {\mathbb C}^{*}$. By taking the limit at $t=0$ we obtain
the thesis of the lemma.
\end{proof}
Next we replace our difference equation with a differential equation.
\begin{pro}
\label{pro:difequ}
  Let $w \in {\mathbb C}\{x,y\} \setminus {\mathfrak m}_{x,y}$  and $\Delta \in {\mathbb C}[[x,y]]$.
Consider a solution $\hat{\Gamma}_{\Delta}$ of $(\log \varphi_{0,w} ) (\hat{\Gamma})= - y(y-x) \Delta$. Then we have
\[ S_{w}(\Delta) = \hat{\Gamma}_{\Delta}(x,x) - \hat{\Gamma}_{\Delta}(x,0) . \]
\end{pro}
\begin{proof}
We define $\Pi \in {\mathbb C}[[x,y]]$ such that
\[ y(y-x) \Pi = \sum_{k \geq 2} \frac{(\log \varphi_{0,w})^{k-2} (y(y-x)\Delta(x,y))}{k!} . \]
The definition is correct since the right hand side belongs to the ideal $(y(y-x))$ of ${\mathbb C}[[x,y]]$.
By using
$\hat{\Gamma}_{\Delta} \circ \varphi_{0,w} =
\hat{\Gamma}_{\Delta} + \sum_{j=1}^{\infty} (\log \varphi_{0,w})^{j}(\hat{\Gamma}_{\Delta})/j!$
and $(\log \varphi_{0,w} ) (\hat{\Gamma}_{\Delta})= - y(y-x) \Delta$ we obtain
\[ \hat{\Gamma}_{\Delta} - \hat{\Gamma}_{\Delta} \circ \varphi_{0,w} =
y (y-x) \Delta + (\log \varphi_{0,w})(y(y-x) \Pi)   . \]
Consider a solution $\hat{\epsilon}_{\Delta} \in {\mathbb C}[[x,y]]$ of
$\hat{\epsilon}  - \hat{\epsilon}  \circ \varphi_{0,w} = y(y-x) \Delta(x,y)$.
Then $\hat{\alpha} = \hat{\Gamma}_{\Delta} - \hat{\epsilon}_{\Delta}$ is a solution of
\[ \hat{\alpha} - \hat{\alpha} \circ \varphi_{0,w} = y(y-x) \frac{\log \varphi_{0,w}}{y(y-x)}[y(y-x) \Pi] . \]
The right hand side is a formal power series. We obtain
\[ (\hat{\Gamma}_{\Delta} - \hat{\epsilon}_{\Delta})(x,x) - (\hat{\Gamma}_{\Delta} - \hat{\epsilon}_{\Delta})(x,0)
= 0 \]
by lemma \ref{lem:izs}.
\end{proof}
\begin{cor}
  Suppose $\log \varphi_{0,w} \in {\mathcal X}_{N} \cn{2}$. Then
$S_{w} ({\mathbb C}\{x,y\})$ is contained in ${\mathbb C}\{x\}$.
\end{cor}
\begin{rem}
Given $\log \varphi_{0,w} \in {\mathcal X}_{N} \cn{2}$ there is no a convergent solution
$\hat{\epsilon}_{\Delta}$
% \in {\mathbb C}\{ x,y \}$
of the equation
$\hat{\epsilon}  - \hat{\epsilon}  \circ \varphi_{0,w} = y(y-x) \Delta(x,y)$ for general
$\Delta \in {\mathbb C}\{x,y\}$.
The divergence of $S_{w}(\Delta)$ is a stronger property than the divergence of every $\hat{\epsilon}_{\Delta}$.
\end{rem}
\section{The induced differential equation}
Let $v \in {\mathbb C}[[x,y]]$. Let
$D_{v}:{\mathbb C}\{x,y\} \to {\mathbb C}[[x]]$ be the operator defined by
$D_{v}(H) = \hat{\epsilon}_{H}(x,x) - \hat{\epsilon}_{H}(x,0)$
where $\hat{\epsilon}_{H} \in {\mathbb C}[[x,y]]$ is a solution of the equation
$\partial \hat{\epsilon} / \partial{y} = v H$. The definition of $D_{v}(H)$
does not depend on the choice of $\hat{\epsilon}_{H}$.
This section is devoted to prove
\begin{pro}
\label{pro:main}
Let $v \in {\mathbb C}[[x,y]]$. If $D_{v}({\mathbb C}\{x,y\}) \subset {\mathbb C}\{x\}$ then
$v$ belongs to ${\mathbb C}\{x,y\}$.
\end{pro}
Fix $\epsilon,\delta>0$. We define the Banach space $B_{\epsilon,\delta}$ whose elements are
the series  $H=\sum_{0 \leq j,k} H_{j,k} {x}^{j} {y}^{k} \in {\mathbb C}[[x,y]]$ such that
\[ ||H||_{\epsilon,\delta} = \sum_{0 \leq j,k} |H_{j,k}| {\epsilon}^{j} {\delta}^{k} < +\infty . \]
We have $B_{\epsilon,\delta} \subset {\mathbb C}\{x,y\}$. Moreover, a function $H \in B_{\epsilon,\delta}$
is holomorphic in $B(0,\epsilon) \times B(0,\delta)$ and continuous in
$\overline{B}(0,\epsilon) \times \overline{B}(0,\delta)$. Given $v$ in  ${\mathbb C}[[x,y]]$
we can define for $j \geq 1$ the linear functionals $D_{v}^{j}:B_{\epsilon,\delta} \to {\mathbb C}$
such that
\[ D_{v}(H) = \sum_{j \geq 1} D_{v}^{j}(H) {x}^{j} \]
for any $H \in B_{\epsilon, \delta}$.
\begin{lem}
Let $v \in {\mathbb C}[[x,y]]$. Then $D_{v}^{j}:B_{\epsilon,\delta} \to {\mathbb C}$ is a
continuous linear functional for any $j \in {\mathbb N}$.
\end{lem}
\begin{proof}
We denote $H=\sum_{0 \leq k,l} H_{k,l}(H) {x}^{k} {y}^{l}$. Denote
$v= \sum_{a,b \geq 0} v_{a,b} x^{a} y^{b}$. A solution of
\[ \frac{\partial \hat{\epsilon}}{\partial y} = v H = \sum_{a,b,k,l \geq 0} v_{a,b} H_{k,l}(H) x^{a+k} y^{b+l} \]
is obtained by making
$\hat{\epsilon}  = \sum_{a,b,k,l \geq 0}  (b+l+1)^{-1} v_{a,b} H_{k,l}(H) x^{a+k} y^{b+l+1}$.
Hence we have
\[ D_{v}(H) = \sum_{a,b,k,l \geq 0} \frac{v_{a,b} H_{k,l}(H)}{b+l+1} x^{a+k+b+l+1} . \]
and then
\[ D_{v}^{j} = \sum_{k+l<j} \left({ \sum_{a+b=j-k-l-1} \frac{v_{a,b}}{b+l+1} }\right) H_{k,l}
\ \ \forall j \in {\mathbb N} .\]
It suffices to prove that $H_{k,l}: B_{\epsilon,\delta} \to {\mathbb C}$ is a continuous functional
for all $0 \leq k,l$.  Since $|H_{k,l}(H)| \epsilon^{k} \delta^{l} \leq ||H||_{\epsilon, \delta}$
for any $H \in B_{\epsilon, \delta}$ then we obtain
$||H_{k,l}|| = \sup \{ |H_{k,l}(H)|: H \in B_{\epsilon,\delta} \ {\rm with} \ ||H||_{\epsilon,\delta} \leq 1\}
\leq {\epsilon}^{-k} {\delta}^{-l}$.
\end{proof}
\begin{lem}
\label{lem:boupri}
Let $v \in {\mathbb C}[[x,y]]$. Consider the operators $D_{v}^{j}: B_{\epsilon, \delta} \to {\mathbb C}$
for $j \in {\mathbb N}$. Either
$\lim \sup_{j \to \infty} \sqrt[j]{||D_{v}^{j}||} < + \infty$ or
$D_{v}(H) \not \in {\mathbb C}\{x\}$ for any $H$ in a dense subset of $B_{\epsilon,\delta}$.
\end{lem}
\begin{proof}
  Suppose $\lim \sup_{j \to \infty} \sqrt[j]{||D_{v}^{j}||} = + \infty$. We choose a sequence
$(a_{j})$ of positive numbers such that $a_{j} \to \infty$ and
\[ \lim \sup_{j \to \infty} \frac{\sqrt[j]{||D_{v}^{j}||}}{a_{j}} = + \infty. \]
Hence $\lim \sup_{j \to \infty} ||D_{v}^{j} /a_{j}^{j}|| = + \infty$. We deduce that
\[ \lim \sup_{j \to \infty} |D_{v}^{j}(H)| / a_{j}^{j} = + \infty \]
for any $H$ in a dense subset $E$ of $B_{\epsilon,\delta}$ by the uniform boundedness principle.
Moreover, since
\[ \lim \sup_{j \to \infty} \sqrt[j]{|D_{v}^{j}(H)|} \geq \lim \inf_{j \to \infty} a_{j} = + \infty \]
then $D_{v}(H) \not \in {\mathbb C}\{ x \}$ for any $H \in E$.
\end{proof}
\begin{pro}
Consider $v \in {\mathbb C}[[x,y]]$ and $D_{v}: B_{\epsilon, \delta} \to {\mathbb C}[[x]]$.
Suppose $D_{v}(B_{\epsilon,\delta}) \subset {\mathbb C}\{x\}$.
Then there exists $\eta_{\epsilon,\delta} >0$ such that
$D_{v}(H)$ belongs to ${\mathcal O}(B(0,\eta_{\epsilon,\delta}))$ for any $H \in B_{\epsilon,\delta}$.
\end{pro}
\begin{proof}
   There exists $\eta_{\epsilon,\delta}>0$ such that
$\lim \sup_{j \to \infty} \sqrt[j]{||D_{v}^{j}||} \leq 1 / \eta_{\epsilon,\delta}$
by lemma \ref{lem:boupri}. As a consequence
\[ \lim \sup_{j \to \infty} \sqrt[j]{|D_{v}^{j}(H)|} \leq
\lim \sup_{j \to \infty} \left({ \sqrt[j]{||D_{v}^{j}||} \sqrt[j]{||H||_{\epsilon,\delta}} }\right)
\leq  1 / \eta_{\epsilon,\delta} . \]
That implies that $D_{v}(H) \in {\mathcal O}(B(0,\eta_{\epsilon,\delta}))$ for any $H \in B_{\epsilon,\delta}$.
\end{proof}
\begin{proof}[Proof of prop. \ref{pro:main}]
Consider the functionals $D_{v}:B_{1,1} \to {\mathbb C}[[x]]$ and
$D_{v}^{j}:B_{1,1} \to {\mathbb C}$ for any $j \geq 1$.
Since $D_{v}(B_{1,1}) \subset {\mathbb C} \{ x \}$ then there exists $C \geq 1$ such that
$||D_{v}^{j}|| \leq C^{j}$ for any $j \geq 1$ by lemma \ref{lem:boupri}.
We denote $v = \sum_{0 \leq a,b} v_{a,b} {x}^{a} {y}^{b}$. We have
\[ D_{v}^{1}(1)=v_{0,0} \Rightarrow |v_{0,0}| \leq ||D_{v}^{1}|| ||1||_{1,1} \leq C. \]
We want to estimate $v_{k,0}$, $v_{k-1,1}$, $\hdots$, $v_{0,k}$ for any $k \geq 0$.
Let us calculate $D_{v}(y^{r-1})$ for $r \in {\mathbb N}$.
The equation $\partial \hat{\epsilon} / \partial y = v y^{r-1} =
\sum_{a,b \geq 0} v_{a,b} x^{a} y^{b+r-1}$ has a solution
$\hat{\epsilon} = \sum_{a,b \geq 0} (b+r)^{-1} v_{a,b} x^{a} y^{b+r}$.
We deduce that
\[ D_{v}(y^{r-1})= \sum_{a,b \geq 0} (b+r)^{-1} v_{a,b} x^{a+b+r} . \]
In particular $D_{v}^{k+r}(y^{r-1})= \sum_{a+b=k} v_{a,b}/(b+r) = \sum_{b=0}^{k} v_{k-b,b}/(b+r)$
for any $r \in {\mathbb N}$. We obtain
\[ \mbox{Hilb}^{k}
\left({
\begin{array}{c}
v_{k,0} \\
v_{k-1,1} \\
\vdots \\
v_{0,k}
\end{array}
}\right)
=
\left({
\begin{array}{c}
D_{v}^{k+1}(1) \\
D_{v}^{k+2}(y) \\
\vdots \\
D_{v}^{2k+1}(y^{k})
\end{array}
}\right)  \]
where $\mbox{Hilb}^{k}$ is the $(k+1) \times (k+1)$ Hilbert matrix; this is a real symmetric matrix such that
$\mbox{Hilb}_{a,b}^{k} = 1 /(a+b-1)$ for $1 \leq a,b \leq k+1$.
%
%The Hilbert matrix is the matrix associated to the bilinear form
%\[ <P,Q> = \int_{0}^{1} P(r)Q(r) dr \]
%in the basis $1$, $\hdots$, $x^{k}$ of the space of real polynomials in one variable of degree at
%most $k$. Therefore $Hil^{k}$ is not singular and all its eigenvalues are positive numbers.
%In order to estimate $||v_{k,0}, \hdots, v_{0,k}||_{2}$ we want to estimate the spectral
%norm of the inverse of $Hil^{k}$, i.e. $||{(Hil^{k})}^{-1}||_{2}$.
%Since ${(Hil^{k})}^{-1}$ is hermitian then
%\[ ||{(Hil^{k})}^{-1}||_{2} = \max eigenvalues ({(Hil^{k})}^{-1}) =
%\frac{1}{\min eigenvalues (Hil^{k})} . \]
%Let $\rho =1 + \sqrt{2}$ and $K=(8 \pi^{3/2} 2^{3/4})/{(1+ \sqrt{2})}^{4}$; we obtain
%\[ ||{(Hil^{k})}^{-1}||_{2} = \frac{{\rho}^{4k}}{K \sqrt{k}} (1+o(1)) \]
%as $k \to \infty$ \cite{Kalyabin}.
%
Moreover $\mbox{Hilb}^{k}$ is positive definite and following \cite{Kalyabin} we obtain that
\[ ||{(\mbox{Hilb}^{k})}^{-1}||_{2} = \frac{{\rho}^{4k}}{K \sqrt{k}} (1+o(1)) \]
where $K=(8 \pi^{3/2} 2^{3/4})/{(1+ \sqrt{2})}^{4}$, $\rho =1 + \sqrt{2}$ and
${||\hdots||}_{2}$ is the spectral norm.
We have $|D_{v}^{k+l+1}(y^{l})| \leq ||D_{v}^{k+l+1}|| ||{y}^{l}||_{1,1} \leq C^{k+l+1}$.
As a consequence we obtain
\[ ||v_{k,0}, \hdots, v_{0,k}||_{2} \leq \frac{{\rho}^{4k}}{K \sqrt{k}} \sqrt{k+1} C^{2k+1} (1+o(1)) . \]
where $||v_{k,0}, \hdots, v_{0,k}||_{2}$ is the euclidean norm. Then
\[ |v_{l,m}| \leq \frac{{\rho}^{4(l+m)}}{K \sqrt{l+m}} \sqrt{l+m+1} C^{2(l+m)+1} (1+o(1)) \]
for $0 \leq l,m$ where $\lim_{l+m \to \infty} o(1)=0$.
We deduce that $v$ belongs to
${\mathcal O}(B(0,1/({\rho}^{4} {C}^{2})) \times B(0,1/({\rho}^{4} {C}^{2})) )$.
\end{proof}
\section{End of the proof of the Main Theorem}
The following proposition will imply the Main Theorem.
\begin{pro}
\label{pro:main2}
Let $w \in {\mathbb C} \{x,y\} \setminus {\mathfrak m}_{x,y}$. Suppose that
$\log \varphi_{0,w}$ is not convergent.
Then there exists $\Delta \in {\mathfrak m}_{x,y}$ such that
$\varphi_{\Delta,w}$ has non-embeddable formal class.
\end{pro}
\begin{proof}
Suppose the result is false. Hence
$S_{w}({\mathfrak m}_{x,y}) \subset {\mathbb C} \{x\}$ by
proposition \ref{pro:proequ}. Let $\hat{w} \in {\mathbb C}[[x,y]]$ be the
unit such that $L_{0,w} = \hat{w}(x,y) \partial / \partial{y}$
(see lemma \ref{lem:pary}). By hypothesis $\hat{w}$ is a divergent power series.
By proposition \ref{pro:difequ} the series
$\hat{\Gamma}_{\Delta}(x,x) - \hat{\Gamma}_{\Delta}(x,0)$ belongs to ${\mathbb C}\{x\}$
for any solution $\hat{\Gamma}_{\Delta} \in {\mathbb C} [[x,y]]$ of
\[ \frac{\partial \hat{\Gamma}}{\partial{y}} = - \frac{\Delta(x,y)}{\hat{w}(x,y)} \]
and any $\Delta \in  {\mathfrak m}_{x,y}$.
Since $D_{-x/\hat{w}}({\mathbb C}\{x,y\}) \subset {\mathbb C}\{x\}$ then
$-x /\hat{w}$ belongs to ${\mathbb C} \{x,y\}$ by proposition \ref{pro:main}. We deduce that
$\hat{w} \in {\mathbb C} \{x,y\}$; that is a contradiction.
\end{proof}
To end the proof of the Main Theorem it suffices to exhibit an example of a
diffeomorphism $\varphi_{0,w}$ such that $\log \varphi_{0,w}$ is divergent by
proposition \ref{pro:main2}.

If $w(0,y) \in {\mathcal O}({\mathbb C})$ then $(y \circ \varphi_{0,w})(0,y)$ is an entire function
different than $y$. Then ${(\log \varphi_{0,w})}_{|x=0}$ is divergent (see \cite{Ah-Ro}).
Therefore $\log \varphi_{0,w}$ is divergent. In particular we can choose $w=1$.
\section{Remarks and generalizations}
  In our approach a unipotent $\tau \in \diff{}{n}$ has embeddable formal class
if $\log \tau$ is formally conjugated to a germ of convergent vector field.
This is a strong concept of embeddability. There is an alternative definition:
We say that the formal class of $\tau \in \diff{}{n}$ is weakly embeddable if there exists a germ of
vector field $Y$ vanishing at $0$ whose exponential is formally conjugated to $\tau$.
This definition is suppler but not so geometrically significant.
For instance $Id = {\rm exp}(0) \in \diff{}{}$ is the exponential of
every germ of vector field whose first jet is $2 \pi i z \partial / \partial{z}$.
It is natural to restrict our study to the strong case.
Anyway, in the family $(\varphi_{\Delta,w})$ the strong and weak concepts of
embeddability for formal classes coincide. Hence the formal class of a general $\varphi_{\Delta,w}$
is non-weakly embeddable.

 There exists $\varphi_{\Delta_{0},w_{0}}$ whose formal class is non-embeddable.
That is the generic situation. Consider the set
\[ E = \{ \varphi_{\Delta,w} : \Delta \in {\mathfrak m}_{x,y}  \ {\rm and} \  w(0,0)=1 \} \subset \diff{u}{2} . \]
For every $\varphi_{\Delta,w}$ there exists $\mu \in {\mathbb C}^{*}$
such that $(x/\mu,y/\mu) \circ \varphi_{\Delta,w} \circ (\mu x, \mu y)$ is in $E$.
Then to study the embeddability of formal classes in the family $(\varphi_{\Delta,w})$ we can restrict
ourselves to $E$. In particular we can suppose $\varphi_{\Delta_{0},w_{0}} \in E$. The set $E$ is an affine
space whose underlying vector space is ${\mathfrak m}_{x,y} \times {\mathfrak m}_{x,y}$. Then $E$ is the
union of the complex lines
\[ L_{A,B}: \lambda \mapsto \varphi_{\Delta_{0}+\lambda A, w_{0} + \lambda B} .  \]
where $A,B \in {\mathfrak m}_{x,y}$.
By arguing like in section \ref{sec:polfam} and applying proposition \ref{pro:PM} we can prove that for any
line $L_{A,B}$ through $\varphi_{\Delta_{0},w_{0}}$ the transport mapping
is divergent outside of a polar set. The non-embeddability of the formal class
is clearly the generic situation in $E$.

  Consider $\Delta, w$ such that $\varphi_{\Delta, w}$ has non-embeddable formal class.
We define
\[ \varphi_{\Delta, w}^{n}  = (z_{1} + z_{2}(z_{2}-z_{1}) \Delta(z_{1},z_{2}),
z_{2} + z_{2}(z_{2}-z_{1}) w(z_{1},z_{2}), z_{3},\hdots,z_{n}) . \]
Then  $\varphi_{\Delta, w}^{n} \in \diffh{u}{n} \cap \diff{}{n}$ has non-embeddable formal class for any
$n \geq 2$. Hence there are unipotent germs of diffeomorphisms with non-embeddable formal class
for any dimension greater than $1$.

Let $f \in  {\mathfrak m}_{x,y}$. Consider the family
\[ \varphi_{\Delta, w}^{f} = (x + f(x,y) \Delta(x,y), y + f(x,y) w(x,y))  \]
where $\Delta (0,0)=0 \neq w(0,0)$. The choice  $f=y(y-x)$ is by no means special. We can choose $f$
such that its decomposition ${x}^{m} f_{1}^{n_{1}} \hdots f_{p}^{n_{p}}$ in irreducible
factors satisfies $\nu((f_{1} \hdots f_{p})(0,y)) > 1$. This condition means that
in a suitable domain $\sharp (\{f=0\} \cap \{x=c\}) >1$ for any $c \neq 0$ in a neighborhood of $0$.
It is the condition we need to define an analogue of the transport mapping.
Fix $w$ such that $\ln \varphi_{0,w}^{f}$ is divergent;
we can adapt the results in this paper to prove that there exists $\varphi_{\Delta, w}^{f}$
with non-embeddable formal class for some
$\Delta \in {\mathfrak m}_{x,y}$. The two main difficulties in the proof are:
\begin{itemize}
\item The formal invariants are slightly more complicated \cite{UPD} \cite{nota}.
This phenomenon is isolated in the example $f= {y}^{a_{0}} {(y-x)}^{a_{1}}$ for
$a_{j} \in {\mathbb N}$. If $a_{j}>1$ the function ${|Jac \  \varphi_{\Delta, w}^{f}|}_{|y=jx}$
is identically equal to $1$, it is a trivial formal invariant.
Anyway, there are always non-constant functions on $y=0$ and $y=x$
which are formal invariants. This is crucial to prove proposition \ref{pro:tracon}
since otherwise we can not claim that the action of a formal conjugation on
a fixed points curve is convergent. The rest of the proof is basically the same.

\item The technical details in the proofs are in general trickier.
That is the situation if the curve $f=0$ is complicated, for instance if its
components are singular.
Anyway, the proof basically follows the same lines. The additions are intended
to make the strategy in this paper work. We chose the case $f=y(y-x)$ because the
presentation is clearer but it contains all the main ideas.
\end{itemize}

%\nocite{*}
\bibliographystyle{plain}
\bibliography{rendu}
\end{document}